\documentclass[12pt,reqno]{amsart}
\usepackage{amssymb, amsmath,amsthm}
\usepackage{graphicx}
\usepackage{subfig}
\newtheorem{thm}{Theorem}

\numberwithin{defn}{section}
\numberwithin{thm}{section}
\numberwithin{Lemma}{section}
\numberwithin{Corollary}{section}
\numberwithin{Example}{section}
\numberwithin{subsection}{section}
\numberwithin{Remark}{section}
\numberwithin{equation}{section}
\numberwithin{ppn}{section}
\begin{document}
\title[ A family of  optimal eight-order iterative  ... ]
{An efficient  family of optimal eight-order  iterative methods for solving nonlinear equations} 
\author{Anuradha Singh and J. P. Jaiswal }
\date{}
\maketitle


\textbf{Abstract.} 
The prime objective of this paper is to  design  a new family of eighth-order iterative methods  by accelerating the order of convergence and efficiency index of well existing seventh-order iterative method of \cite{Soleymani1}  without using more function evaluations  for finding simple roots of nonlinear equations. The presented iterative family requires three function and one derivative evaluations and thus agrees with  the conjecture of Kung-Traub for the case $n = 4$  (i.e. optimal).  
We have also discussed  the derivative free version of the proposed scheme. 
Numerical comparisons have been carried out to demonstrate the efficiency and the performances of proposed method. 
\section{Introduction}
The Newton's iterative method is one of the prominent methods for finding roots of a nonlinear equation 
\begin{equation*}
f(x)=0. 
\end{equation*}
 It is well known that the order of convergence of the Newton's method is two. In real life problems, the evaluation of derivatives is difficult (sometimes not possible) or takes up a very long computational time, in that case it is hard to implement Newton method. To overcome this problem, Steffensen has provided an iterative method
\begin{equation}
x_{n+1}=x_n-\frac{f(x_n)^2}{f(x_n+f(x_n))-f(x_n)},
 \end{equation}
with two function evaluations and the same convergence rate as Newton-Raphson's method.
Solving nonlinear equations is one of the most important and gripping task in numerical analysis. The vast literature is available on the solution of nonlinear equations or system of nonlinear equations, one may refer \cite{Chun}-\cite{Traub1}. Very recently, in \cite {Petkovic} Petkovic et al. have provided detail discussion on multipoint methods for solving nonlinear equations. Such type of  schemes have drawn the attention of many researchers. In recent past, many researchers have focused to optimize the existing methods without evaluating additional functions and first derivative of functions. 

Recently, Soleymani et al. has established seventh-order method defined in \cite{Soleymani1} is given by 
\begin{eqnarray}\label{eqn:11}
y_n&=&x_n-  \frac{f(x_n)}{f'(x_n)}\nonumber\\
z_n&=&y_n- \frac{f(y_n)}{f[x_n,y_n]}.{G(t_n)}\nonumber\\
x_{n+1}&=&z_n- \frac{f(z_n)}{f[y_n,z_n]}.{H(t_n)},
\end{eqnarray}
where $t_n=\frac{f(y_n)}{f(x_n)}$ and  $G(0)=G'(0)=1$, $ \left|G''(0)\right|\leq +\infty$; $H(0)=1, H'(0)=0,H''(0)=2$,  $ \left|H^{(3)}(0)\right|\leq +\infty$. 

To compare efficieny of  different  iterative methods the efficiency index is defined in \cite{Gautschi,Traub2} and given by $p^{1/n}$, where $p$ is the order of convergence and  $n$ be the number of function evaluations  of the iterative  method.  Kung and Traub \cite{Kung} presented a hypothesis on the optimality of the iterative methods by giving $2^{n-1}$ as the optimal order. Thus the efficiency index of the method $(\ref{eqn:11})$ is $ 7^{1/4}\approx 1.626$ and clearly it is not optimal (because this method requires four function evaluations (three functions and one derivative) so for optimal its order of convergence should be $ 2^3=8$). The motive of this paper is to accelaerate the order of converegence of the method (1.2) from seven to eight without adding more evaluations, and thus it will agrees with Kung-Traub conjecture as well as give higher efficiency index.

The rest of the paper is organized as follows: in section 2, we propose a new optimal eight-order iterative method for finding simple roots of nonlinear equations . Particular case of  proposed iterative family have also been given. In section 3 an approach has been given to make our proposed method derivative free . In section 4, we employ some numerical examples to compare the performance of our new method  with some existing eight-order methods. 
 Finally, in the last  section  we furnished the concluding remarks and future work.
 
\section{ Improved Scheme and Convergence Analysis }
In, this section, the order of convergence of the method $(\ref{eqn:11})$ will be accelerated from seven to eight to make it optimal. The order of convergence of the method $(\ref{eqn:11})$ is seven by using four function [$f(x_n)$, $f'(x_n)$, $f(y_n)$, $f(z_n)$] evaluations, which is clearly not optimal.
To build an optimal eight-order method family of iterative methods without using more evaluations of the functions, we consider the following family
\begin{eqnarray}\label{eqn:21}
y_n&=&x_n- A(t_1). \frac{f(x_n)}{f'(x_n)}\nonumber\\
z_n&=&y_n- B(t_2) .\frac{f(y_n)}{f[x_n,y_n]}\nonumber\\
x_{n+1}&=&z_n- \{P(t_2)+Q(t_3)+R(t_4)\}.\frac{f(z_n)}{f[y_n,z_n]},
\end{eqnarray}
where  $t_1=\frac{f(x_n)}{f'(x_n)}$, $t_2=\frac{f(y_n)}{f(x_n)}$,  $t_3=\frac{f(z_n)}{f(y_n)}$  and $t_4=\frac{f(z_n)}{f(x_n)}$. The weight functions should be chosen such that the order arrives at optimal level eight. Theorem (2.1) gives the conditions on weight functions to reach at optimal level of convergence.
  \\                                                                
\begin{thm}
Let the function $f: D\subseteq \Re \longrightarrow \Re$  have sufficient number of continuous derivatives in a neighborhood $D$ of simple root  $\alpha$ of f. Then the method defined by $(\ref{eqn:21})$ has eighth-order convergence, when the weight functions $A(t_1)$,   $B(t_2)$,  $P(t_2)$,  $Q(t_3)$ and  $R(t_4)$, satisfy the following conditions:
\begin{eqnarray}\label{eqn:22a}
&& A(0)=1,\  A^{'}(0)=0,\ A^{''}(0)=0, \  \left|A^{(3)}(0)\right|\leq +\infty, \nonumber\\
&& B(0)=1,\  B^{'}(0)=1,\  \left|B^{(3)}(1)\right|\leq +\infty, \nonumber\\
&& R(0)=1-P(0)-Q(0), \nonumber\\
&& P'(0)=0,\ P''(0)=2, \ P^{(3)}(0)=6B''(0)-12, \  \left|P^{(4)}(0)\right|\leq +\infty, \nonumber\\
&& Q'(0)=0, \  \left|Q{''}(0)\right|\leq +\infty, \nonumber\\
&& R'(0)=2,\  \left|R{''}(0)\right|\leq +\infty.
\end{eqnarray}
\end{thm}

\begin{proof}
Let $e_n=x_n-\alpha$ be the error in the $n^{th}$ iterate and $c_h=\frac{f^{(h)}(\alpha)}{h!}$, $h=1,2,3 . . .$. We provide the Taylor's series expansion of each term involved in $(\ref{eqn:21})$. By Taylor expansion around the simple root in the $n^{th}$ iteration, we have\\
\begin{equation}\label{eqn:23a}
\begin{split}
f(x_n)=c_1e_n+c_2e_n^2+. . . +O(e_n^{10}),       
\end{split}
\end{equation}
and
\begin{equation}\label{eqn:24a}
\begin{split}
f'(x_n)=c_1+2c_2e_n+. . . +O(e_n^{9})                     .  
\end{split}
\end{equation}
Furthermore, it can be easily find 
\begin{equation}\label{eqn:25}
\frac{f(x_n)}{f'(x_n)}=e_n-\frac{c_2e_n^2}{c_1}+. . . +O(e_n^9).
\end{equation}
By considering this relation and $A(0)=1$  we obtain
\begin{eqnarray}\label{eqn:26a}
y_n&=&\alpha+\left\{\frac{c_2}{c_1}-A'(0)\right\}e_n^2+\left\{-\frac{2c_2^2}{c_1^2}+ \frac{2(c_3+c_2A'(0))}{c_1}-\frac{A''(0)}{2}\right\}e_n^3 \nonumber\\
&&+. . . +O(e_n^{9}).
\end{eqnarray}
At this time, we should expand $f'(y_n)$ around the root by taking into consideration $(\ref{eqn:26a})$. Accordingly, we have
\begin{eqnarray}\label{eqn:27}
f(y_n)&=&(c_2-c_1A'(0))e_n^2+\left\{-\frac{2c_2^2}{c_1}+2(c_3+c_2A'(0))-\frac{1}{2}c_1A''(0)\right\}e_n^3 \nonumber\\
&&+. . . +O(e_n^{9}),
\end{eqnarray}

\begin{eqnarray}\label{eqn:27b}
\frac{f(y_n)}{f(x_n)}&=&\left\{\frac{c_2}{c_1}-A'(0)\right\}e_n+\left\{-\frac{3c_2^2}{c_1^2}+\frac{2c_3+3c_2A'(0)}{c_1}-\frac{A''(0)}{2}\right\}e_n^2 \nonumber\\
&&+. . . +O(e_n^{9}),
\end{eqnarray}

and
\begin{equation}\label{eqn:27a}
f[x_n,y_n]=c_1+c_2e_n+\left\{ \frac{c_2^2}{c_1}+c_3-c_2A'(0)\right\}e_n^2+. . . +O(e_n^{9}).
\end{equation}
Using the equations $(\ref{eqn:27})$, $(\ref{eqn:27a})$,   $(\ref{eqn:27b})$  and $A'(0)=0$, $B(0)=1$, $B'(0)=1$,  in the second step of $(\ref{eqn:21})$, we can find

\begin{eqnarray}\label{eqn:28}
&&z_n\nonumber\\
&&=\alpha+\frac{c_2(c_1(-2c_3+c_1A''(0))-c_2^2(-6+B''(0)))}{2c_1^3}e_n^4\nonumber\\
&&+\frac{1}{12c_1^4}(-3c_1(-4c_3+c_1A''(0))(-2c_3+c_1A''(0))+3c_1c_2^2(-4c_3+c_1A''(0))\nonumber\\
&& (-20+3B''(0))+2c_2^2c_2(-12c_4+c_1A^{(3)}(0))+c_2^4(54(-4+B''(0)-2B^{(3)}))e_n^5\nonumber\\
&&+ . . .+O(e_n^{9}).
\end{eqnarray}
 By virtue of the above equation, we have
\begin{eqnarray}\label{eqn:29}
&&f(z_n)\nonumber\\
&&=\frac{c_2(c_1(-2c_3+c_1A''(0))-c_2^2(-6+B''(0)))}{2c_1^2}e_n^4\nonumber\\
&&+\frac{1}{12c_1^3}(-3c_1(-4c_3+c_1A''(0))(-2c_3+c_1A''(0))+3c_1c_2^2(-4c_3+c_1A''(0))\nonumber\\
&&(-20+3B''(0)) +2c_2^2c_2(-12c_4+c_1A^{(3)}(0))+c_2^4(54(-4+B''(0)-2B^{(3)}))e_n^5\nonumber\\
&&+ . . .+O(e_n^{9}).
\end{eqnarray}
\begin{eqnarray}\label{eqn:29a}
f[y_n,z_n]=c_1+ \frac{c_2^2}{c_1}e_n^2+. . . +O(e_n^{9}).
\end{eqnarray}
\begin{eqnarray}\label{eqn:210}
\frac{f(z_n)}{f(y_n)}&=&\frac{c_1(-2c_3+c_1A''(0))-c_2^2(-6+B''(0))}{c_1}e_n^2\nonumber\\
&& +\frac{1}{6c_1^3}\{3c_1c_2(-4c_3(-6+B''(0))+c_1A''(0)(-5+B''(0)))\nonumber\\
&&+c_1^2(-12c_4+c_1A^{(3)}(0))+c_2^3(-72+21B''(0)-B^{(3)}(0))\}e_n^3 \nonumber\\
&&+ . . . +O(e_n^{9}).
\end{eqnarray}

\begin{eqnarray}\label{eqn:211}
\frac{f(z_n)}{f(x_n)}&=&\frac{c_2(c_1(-2c_3+c_1A''(0))-c_2^2(-6+B''(0)))}{2c_1^3}e_n^3\nonumber\\
&& +\frac{1}{12c_1^4}\{-3c_1^2(-4c_3+c_1A''(0)) (-2c_3+c_1A''(0))\nonumber\\
&&+3c_1c_2^2(-12c_3(-7+B''(0))+c_1A''(0)(-22+3B''(0)))\nonumber\\
&&+2c_1^2c_2(-12c_4+c_1A^{(3)}(0))+ c_2^4(60B''(0)-2(126+B^{(3)}(0))\}e_n^3 \nonumber\\
&&+ . . . +O(e_n^{9}).
\end{eqnarray}
Finally, using $(\ref{eqn:27b})$, $(\ref{eqn:210})$, $(\ref{eqn:211})$, $(\ref{eqn:29})$, $(\ref{eqn:29a})$  and $ R(0) = 1 - P(0) - Q(0)$ ,  $P'(0) = 0$,  $Q'(0) = 0$, $P''(0) = 2$, $R'(0) = 2$, $ A''(0) = 0$, $P^{(3)}(0) = 6B''(0) - 12$ in the last step of $(\ref{eqn:21})$, we get the final error expression which is given by
\begin{eqnarray}\label{eqn:212}
e_{n+1}&=&\frac{c_2}{48c_1^7}\{2c_1c_3+c_2^2(-6+B''(0))\}\nonumber\\ &&\{12c_1^2c_3^2Q''(0)+12c_1c_2^2c_3(8+(-6+B''(0))Q''(0))\nonumber\\
&&+4c_1^2c_2(-6c_4+c_1A^{(3)}(0))+ c_2^4(108Q''(0)+3B''(0) \nonumber\\
&& (8+(-12+B''(0))Q''(0))-8(9+B^{(3)}(0)+P^{(4)}(0) ) )\}e_n^8+O(e_n^{9}). \nonumber\\
\end{eqnarray}
Thus, theorem is proved.
\end{proof}

\textbf{Particular Case:}\\
Let
\begin{eqnarray*}
  A(t_1)&=&1+ \alpha\  t_1^3,  \nonumber\\ 
 B(t_2)&=&1+t_2+ \beta \  t_2^2, \nonumber \\
  C(t_2)&=&t_2^2+ 2(\beta-1) \  t_2^3, \nonumber\\ 
  D(t_3)&=&\gamma \ t_3^2, \nonumber \\
 E(t_4)&=&1+2\  t_4+\delta \  t_4^2, 
\end{eqnarray*}
where $\alpha,\ \beta,\ \gamma\ , \  \delta \in \mathbb{R}$, then the method becomes 

\begin{eqnarray}\label{eqn:213}
y_n&=&x_n- \{1+ \alpha\  t_1^3\} \frac{f(x_n)}{f'(x_n)},\nonumber\\
z_n&=&y_n- \{1+t_2+ \beta \  t_2^2\}\frac{f(y_n)}{f[x_n,y_n]},\nonumber\\
x_{n+1}&=&z_n- \{\{t_2^2+ 2(\beta-1) \  t_2^3\}+\{\gamma \ t_3^2\}+\{1+2\  t_4+\delta \  t_4^2\}\}.\frac{f(z_n)}{f[y_n,z_n]}.\nonumber\\
\end{eqnarray}
Clearly, this method is four-parametric where  $t_1=\frac{f(x_n)}{f'(x_n)}$, $t_2=\frac{f(y_n)}{f(x_n)}$,  $t_3=\frac{f(z_n)}{f(y_n)}$  and $t_4=\frac{f(z_n)}{f(x_n)}$.
Then its error expression becomes

\begin{eqnarray}\label{eqn:212}
e_{n+1}&=&\frac{c_2}{c_1^7}\{c_1c_3+c_2^2(-3+\beta)\}\nonumber\\ &&\{\alpha\ c_1^2c_2+(-3+2\beta+(-3+\beta^2)\gamma)c_2^4\nonumber\\
&&+2(2+(-3+\beta)\gamma)c_1c_2^2c_3+c_1^2(\gamma c_3^2-c_2c_4)
 \}e_n^8+O(e_n^{9}). \nonumber\\
\end{eqnarray}
\textit{Remark:} By taking different values of $\alpha$, $\beta$, $\gamma$  and  $\delta$ one can get a number of eighth-order iterative methods. Our class of three-step  method requires four evaluations (one derivative and three function) and has the order of convergence eight. Therefore our class is of optimal order and support the Kung-Traub conjecture for $n=4$. Clearly its efficiency index is  $8^{\frac{1}{4}} \approx 1.682$  which is more than efficiency index $7^{\frac{1}{4}} \approx 1.626$ of method (1.2).

\section{Derivative-free Scheme}
In the real world problems of science and engineering sometimes the derivative of the function is not easy to calculate 
or time consuming. To overcome this problem, in recent days many researches have focused to develop derivative-free methods to solve real world problems e.g. \cite{Soleymani4}, \cite{Soleymani5} and \cite{carlos} . 
In this section,  we give the derivative free version of the  method of previous section.

To do so, we replace $f'(x_n) \approx f[w_n,x_n]$  in the equation $(\ref{eqn:21})$  where $w_n=x_n+ f(x_n)$, then the method  becomes 
\begin{eqnarray}\label{eqn:310}
y_n&=&x_n- A(t_1). \frac{f(x_n)}{f[w_n,x_n]},\nonumber\\
z_n&=&y_n- B(t_2) .\frac{f(y_n)}{f[x_n,y_n]},\nonumber\\
x_{n+1}&=&z_n- \{P(t_2)+Q(t_3)+R(t_4)\}.\frac{f(z_n)}{f[y_n,z_n]},
\end{eqnarray}
and it can be seen that the error equation under the same conditions on weight functions as of theorem (2.1) is given by 
\begin{eqnarray}\label{eqn:311}
e_{n+1}&=&\frac{(1+ c_1)c_2^4(-2+ c_1(-2+Q''[0]))e_n^5}{2c_1^2}\nonumber\\ 
&&+O(e_n^{6}),
\end{eqnarray}
which show fifth-order of convergence. Again to maintain its order of convergence we consider, $w_n=x_n+f(x_n)^2$, then we see that the order of convergence is seven and its error expression is given by
\begin{eqnarray*}\label{eqn:312}
e_{n+1}&=&- \frac{\Bigl(  c_2^3(2  c_1^3 c_2+2c_1c_3+c_2^2(-6+B''[0]))\Bigr)e_n^7}{2c_1^3}\nonumber\\ 
&&+O(e_n^{8}),
\end{eqnarray*}
which also does not meet with our aim. But if we put $w_n=x_n+f(x_n)^3$ then its error equation  (under the same conditions on weight functions as of theorem (2.1)).
\begin{eqnarray}\label{eqn:312}
e_{n+1}&=&\frac{1}{48  c_1^7} c_2(2c_1c_3+c_2^2(-6+B''[0]))\Bigl(-24c_1^5c_2^2+12c_1^2(-2c_2c_4+c_3^2Q''[0]) \nonumber\\
&&+12c_1c_2^2c_3(8+(-6+B''[0])Q''[0])+4c_1^3c_2A^{(3)}[0]+c_2^4(108Q''[0] \nonumber\\
&&+3B''[0](8+(-12+B''[0])Q''[0])-8(9+B^{(3)}[0])+P^{(4)}[0])\Bigr)e_n^8\nonumber\\
&&+O(e_n^{9}).
\end{eqnarray}
Thus the method preserves its order of convergence for $w_n=x_n+f(x_n)^3$. In fact if we put  $w_n=x_n+ \alpha. (f(x_n))^n$, $n \geq 3$ ,where $\alpha \neq 0 \in \mathbb{R}$ in the scheme $(\ref{eqn:310})$ then it gives eighth-order of conevergence.

\section{Results and discussion}   
This section deals with the numerical comparisons of the proposed method $(\ref{eqn:213})$ with $\alpha=\gamma=0$, $\beta=3$,   $\delta=1$. In order to check the effectiveness  of the proposed iterative method we have considered seven test nonlinear functions which are taken from \cite{Babajee}. The test non-linear functions and their roots are listed in Table-1.  In recent days, higher-order methods are very important because numerical applications use high precision computations. Due to this reason all the computations reported have been performed in the programming package $MATHEMATICA\ 8$ using $1000$ digits floating point arithmetic using $''SetAccuraccy"$ command. The results of comparisons are given in Table 2 and Table 3. The computer characteristics during numerical calculations are Microsoft Windows 8 Intel(R) Core(TM)  i5-3210M CPU@ 2.50 GHz with 4.00 GB of RAM, 64-bit Operating System throughout this paper.
 Here we compare performances of our new eighth-order method $(\ref{eqn:213})$ ($OM8$) with the methods of  (34) ($M_{8,1}$), (35) ($ M_{8,2}$) of \cite{taher};  methods NM2 ($M_{8,3}$),  NM3 ($M_{8,4}$) of \cite{Sharma} and methods (11)  ($M_{8,5}$) (15)  ($M_{8,6}$) of \cite{Babajee}.
Table 2 represents the value of  $|f(x_n)|$ calculated for total  number of function evaluations 12 (TNFE-12) for each scheme.  It can be observed from Table 2  in almost cases our method $OM8$ is superior than other methods.
Table 3 exhibits the number of iteration and total number of function evaluation using the stopping criteria $|f(x_{n+1})| < \in$ where $\in = 10^{-50}$. From Table 3, we observe that $OM8$ takes at least equal or less number of iterations for different initial guesses.
\begin{table}[htb]
\caption{Functions and their roots.}
\small
  \begin{tabular}{| lll |} \hline
$f(x)$                                                        &$\alpha$             & \\ \hline 
$f_1(x)=10xe^{-x^2}-1$                      &$\alpha_1 \approx1.67963...$     & \\ 
$f_2(x)=x^5+x^4+4x^2-15$            &$\alpha_2 \approx 1.34742...$     &     \\
$f_3(x)=xe^{x^2}-(Sinx)^2+3Cosx+5$           &$\alpha_3 \approx -1.20764...$          &  \\ 
$f_4(x)=x^4+Sin(\frac{\pi}{x^2})-5$                &$\alpha_4=\sqrt{2}$         & \\ 
$f_5(x)=x^2e^x-Sinx$                      &$\alpha_5=0$              &  \\ 
$f_6(x)=(Sinx-\frac{\sqrt{2}}{2})^2(x+1)$                      &$\alpha_6=-1$       & \\ 
$f_7(x)=Sin3x+x Cosx $ & $\alpha_7 \approx 1.19776...$       &\\ \hline      
  \end{tabular}
  \label{tab:abbr}
\end{table}

\newpage
\begin{table}[!htbp]
\tiny
 \caption{Comparison of absolute value of the functions by different methods after third iteration  (TNFE-12).}
\begin{tabular}{|c |c| c|c | c|  c|  c |c |cc  |}\hline
$\left|f\right|$ & Guess &   $M_{8,1}$&         $M_{8,2}$&    $M_{8,3}$&       $M_{8,4}$&         $M_{8,5}$&   $M_{8,6}$&  $OM8$&\\ \hline      
$\left|f_1\right|$ & 1.72&  0.2e-617&      0.2e-617         & 0.2e-636 &         0.2e-602 &     0.1e-660& 0.2e-654& 0.4e-688&\\
                            &1.5&     0.1e-357&        0.7e-358        &0.7e-365&          0.3e-346&      0.2e-375& 0.8e-361& 0.2e-448&\\
                            &1.7&     0.5e-796&       0.5e-796      &0.1e-794&            0.2e-762&       0.4e-828& 0.5e-809& 0.4e-866&\\  
                           & 1.1&     0.8e-259&      0.7e-257      &0.3e-179&             0.6e-175&       0.1e-205& 0.2e-204& 0.6e-259&\\ \hline
                           
$\left|f_2\right|$&  1.1& NC&                   0.6e-52        &0.1e-177&           0.9e-116&           0.7e-165& 0.5e-299& 0.3e-127&\\
                         &   1.8&   0.3e-148&       0.2e-149       &0.5e-194&           0.2e-175&           0.1e-195& 0.4e-187& 0.1e-225&\\ 
                         &  1.5&    0.6e-347&       0.6e-347      &0.2e-377&           0.1e-348  &          0.9e-437& 0.1e-390& 0.4e-436&\\
                         &  2.0&     0.5e-97&        0.6e-99        &0.3e-148&           0.1e-133&          0.3e-134& 0.4e-132& 0.1e-150&\\ \hline

$\left|f_3\right|$    &-1.1&    0.5e-234&      0.1e-235&         0.4e-337&          0.8e-285&       0.1e-325& 0.6e-433& 0.1e-301&\\
                              &-1.5&    0.5e-124&       0.7e-125&           0.2e-182&        0.2e-161&      0.2e-253&0.2e-205&0.2e-254&\\
                             &-1.0&             div.&         0.2e-45&            0.1e-173&        0.2e-107&      0.3e-158& 0.3e-254&0.9e-116&\\
                            &-1.3&   0.1e-383&      0.1e-383&             0.6e-404&         0.2e-370&      0.2e-411& 0.3e-460&0.1e-468&\\ \hline
%
$\left|f_4\right|$ & 1.0&    0.9e-255&       0.1e-251&      0.2e-226&        0.3e-223&        0.2e-199& 0.5e-227&0.2e-262&\\
                             &1.6&   0.8e-338&       0.7e-338&       0.9e-356&        0.2e-329&         0.3e-370& 0.5e-430&0.1e-441&\\
                            & 1.5&  0.2e-508&        0.2e-508&       0.1e-519&         0.1e-487&       0.9e-526& 0.4e-560&0.3e-532&\\
                            & 2.1& 0.1e-105&         0.2e-107&       0.1e-146&       0.9e-134&        0.1e-144&  0.1e-143& 0.3e-159&\\ \hline

$\left|f_5\right|$ &0.1&   0.1e-272&     0.3e-273&       0.1e-340&       0.4e-303&       0.6e-349& 0.1e-338&0.1e-364&\\
                             &0.5&  0.5e-264&     0.8e-265&      0.1e-346&        0.3e-301&       0.2e-342& 0.7e-382&0.5e-339&\\
                            & -0.1& 0.4e-475&    0.1e-475&       0.2e-470&       0.1e-455&       0.3e-441&  0.2e-436&0.4e-485&\\
                            & -0.5& 0.1e-270&    0.3e-270&      0.8e-237&        0.2e-233&      0.1e-235&  0.3e-233&0.2e-277&\\ \hline

$\left|f_6\right|$ &-0.8&   0.3e-158&      0.1e-162&       0.4e-258&         0.9e-214&       0.4e-248&  0.6e-288&0.3e-254&\\
                             &-1.2&  0.3e-422&       0.4e-423&       0.7e-398&        0.9e-387&       0.3e-385&  0.1e-380&0.2e-435&\\
                            &-0.9&   0.1e-425&        0.1e-425&     0.1e-456&        0.1e-424&        0.4e-500&  0.4e-457&0.6e-526&\\
                            &-1.5&  0.1e-324&         0.8e-324&      0.7e-262&      0.1e-258&         0.7e-273&  0.1e-272& 0.3e-336&\\ \hline

$\left|f_7\right|$ &  1.0&   0.1e-446&      0.2e-448&        0.1e-374&        0.1e-371&       0.2e-378& 0.7e-374&0.2e-505&\\
                             & 0.8&    0.6e-81&                  NC&       0.3e-130&        0.4e-114&       0.7e-159& 0.1e-116&0.6e-150&\\
                            & 1.8&        NC&                      NC&          0.2e-38&          0.3e-20&                 NC&           NC&0.6e-29&\\
                            & 0.3&   0.1e-226&        0.3e-235&       0.8e-244&        0.7e-215&        0.7e-277& 0.1e-278& 0.2e-416&\\ \hline

 \end{tabular}
 \label{tab:abbr}
\end{table}
Here div.= Divergent , I= Indeterminate , NC= Not convergent.

\newpage
\begin{table}[!htbp]
\tiny
 \caption{Comparison of number of iterations and total number of function evaluations (TNFE).}
\begin{tabular}{|c |c| c|c | c|  c|  c |c |cc  |}\hline
$f$ & Guess &   $M_{8,1}$&        $M_{8,2}$&            $M_{8,3}$&    $M_{8,4}$&   $M_{8,5}$&  $M_{8,6}$&  $OM8$&\\ \hline      
$f_1$ & 1.72&  2(8)&                2(8)       &             2(8)&                  2(8)&             2(8)&   2(8)&    2(8)&\\
                            &1.5&    3(12)&                    3(12)&            3(12)&                3(12)&          3(12)&  3(12)&  2(8)&\\
                            &1.7&    2(8)&                       2(8)&              2(8)&                  2(8)&          2(8)& 2(8)& 2(8)&\\  
                           & 1.1&     3(12)&                    3(12)&            3(12)&                3(12)&        3(12)&  3(12)&  2(8)&\\ \hline
                           
$f_2$&  1.1& NC&                3(12)        & 3(12)&          3(12)&          3(12)& 3(12)&  3(12)&\\
                         &   1.8&   3(12)&        3(12)&           3(12)&           3(12)&           3(12)& 3(12)&  3(12)&\\ 
                         &  1.5&    3(12)&        3(12)&           3(12)&            3(12)&           2(8)&  3(12)&  2(8)&\\
                         &  2.0&    3(12)&        3(12)&           3(12)&           3(12)&           3(12)& 3(12)&  3(12)&\\ \hline

$f_3$    &-1.1&     3(12)&                    3(12)&            3(12)&                3(12)&          3(12)&  2(8)&  2(8)&\\
                              &-1.5&      3(12)&        3(12)&           3(12)&           3(12)&           3(12)& 3(12)&  3(12)&\\
                             &-1.0&       div.&                     4(16)&            3(12)&                  3(12)&         3(12)&   3(12)&  3(12)&\\
                            &-1.3&     3(12)&                     3(12)&            3(12)&        3(12)&       2(8)&      2(8)&                  2(8)&\\ \hline
%
$f_4$ & 1.0&         3(12)&        3(12)&           3(12)&           3(12)&           3(12)&                 3(12)&        3(12)&\\
                             &1.6&         3(12)&       3(12)&            3(12)&           3(12)&          3(12)&  2(8)&  2(8)&\\
                            & 1.5&         2(8)&                       2(8)&              2(8)&                  2(8)&          2(8)& 2(8)& 2(8)&\\
                            & 2.1&        3(12)&        3(12)&           3(12)&           3(12)&           3(12)& 3(12)&  3(12)&\\ \hline

$f_5$ &0.1&    3(12)&        3(12)&           3(12)&           3(12)&           3(12)& 3(12)&  3(12)&\\
                             &0.5&   3(12)&        3(12)&           3(12)&           3(12)&           3(12)& 3(12)&  3(12)&\\
                            & -0.1&   2(8)&                2(8)       &             2(8)&                  2(8)&             2(8)&   2(8)&    2(8)&\\
                            & -0.5&   3(12)&        3(12)&           3(12)&           3(12)&           3(12)& 3(12)&  3(12)&\\ \hline

$f_6$ &-0.8&        3(12)&            3(12)&           3(12)&           3(12)&           3(12)&                 3(12)&  3(12)&\\
                             &-1.2&      2(8)&               2(8)&            3(12)&              3(12)&         3(12)&  3(12)&2(8)&\\
                            &-0.9&  2(8)&       2(8)&    2(8)&       3(12)&       2(8)&  2(8)&2(8)&\\
                            &-1.5& 3(12)&        3(12)&           3(12)&           3(12)&           3(12)& 3(12)&  3(12)&\\ \hline

$f_7$ &  1.0& 2(8)&      2(8)&        3(12)&        3(12)&       3(12)& 3(12)&2(8)&\\
                             & 0.8&  3(12)&     NC&      3(12)&       3(12)&      3(12)& 3(12)&3(12)&\\
                            & 1.8&   NC&        NC&         4(16)&          4(16)&                 NC&           NC&4(16)&\\
                           & 0.3&   3(12)&  3(12)&           3(12)&           3(12)&           3(12)& 3(12)&  3(12)&\\ \hline                     
 \end{tabular}
 \label{tab:abbr}
\end{table}

\newpage
\section{Conclusion and Future work}
In this article, we have contributed a new efficient family of eight-order iterative methods to find simple roots of a nonlinear equation    by accelerating the order of convergence and efficiency index of well existing seventh-order iterative method of \cite{Soleymani1}  without using more function evaluations  for finding simple roots of nonlinear equations. Our family requires three function and one derivative evaluations and thus agrees with  the conjecture of Kung-Traub for the case $n = 4$  (i.e. optimal).  An approach to make proposed method free from derivative has also discussed here. Numerical  comparisons  also witness the efficiency of new method. Therefore, we can conclude that the new family is  efficient and give at least equal or better performance over some other eight-order methods. Using the technique of \cite{carlos} other existing methods having derivatives can be made free from derivatives.


\textsc{Anuradha Singh\\
Department of Mathematics, \\
Maulana Azad National Institute of Technology,\\
Bhopal, M.P., India-462051}.\\
E-mail: {singh.anu3366@gmail.com; singhanuradha87@gmail.com}.\\\\
\textsc{Jai Prakash Jaiswal\\
Department of Mathematics, \\
Maulana Azad National Institute of Technology,\\
Bhopal, M.P., India-462051}.\\
E-mail: {asstprofjpmanit@gmail.com; jaiprakashjaiswal@manit.ac.in}.

\begin{thebibliography}{10}
 \bibitem{Soleymani1}
F. Soleymani and B. S. Mousavi: On Novel Classes of Iterative Methods for Solving Nonlinear Equations, Computational Mathematics and Mathematical Physics, 2012, Vol. 52, No. 2, pp. 203-210.
 \bibitem{Sharma}
J. R. Sharma and H. Arora: An efficient family of weighted-Newton methods with optimal eight order convergence, Appl. Math. Lett., 2014, Vol. 29, pp.1-6.
 \bibitem{Babajee}
D. K. R. Babajee, A. Cordero, F. Soleymani and J. R. Torregrosa: On improved three-step schemes with high efficiency index and their dynamics, Numer. Algor., 2014, Vol. 65, pp.153-169.
 \bibitem{Chun}
C. Chun and M. Y. Lee: A new optimal eight-order family of iterative methods for the solution of nonlinear equations, Appl. Math. and Comp., 2013, Vol. 223, pp.506-519.
 \bibitem{Soleymani2}
F. Soleymani: New class of eight-order iterative zero-finders and their basins of attraction, Afr. Mat., 2014, Vol.25, pp.67-79.
 \bibitem{Soleymani3}
A. Cordero, T. Lotfi, K. Mahdiani and J. R. Torregrosa: Two optimal general classes of iterative methods with eight-order, Acta  Appl. Math., 2014, DOI 10.1007/s10440-014-9869-0. 
\bibitem{santiago}
S. Artidiello, A. Cordero, J. R. Torregrosa and M. P. Vassileva : Two weighted eight-order classes of iterative root-finding methods, Int. Jour. of Comp. Mathe. , 2014, DOI 10.1007/s10440-014-9869-0.
\bibitem{Kung}
H. T. Kung and J. F. Traub: Optimal order of one-point and multipoint iteration, JCAM, 21 (1974), 643-651 .
\bibitem{Traub1} 
J. F. Traub: Iterative methods for solution of equations, Chelsea Publishing, New York, USA (1997).
\bibitem{Petkovic}
M. S. Petkovic, B. Neta, L. D. Petkovic and J. Dzunic: Multipoint methods for solving nonlinear equations, Elsevier (2012).
\bibitem{scott}
 M. Scott, B. Neta and C. Chun : Basin Attractors for various methods, Appl. Math. and Comp., 218 (2011), 2584-2599.
\bibitem{neta}
 B. Neta, M. Scott and C. Chun : Basin of attraction for several methods to find simple roots of nonlinear equations, Appl. Math. and Comp., 218(2012),10548-10556.
\bibitem{varona}
J. L. Varona : Graphic and numerical comparison between iterative methods, Math. Intell., 24 (2002), 37-46.
\bibitem{taher}
T. Lotfi, A. Cordero, J. R. Torregrosa, M. A. Abadi and M. M. Zadeh : On generalization based on Bi et al. iterative methods with eight-order convergence for solving nonlinear equations, The Scientific world journal, 2014, Article ID 272949, 8 pages.
\bibitem{Soleymani4}
F. Soleymani : Optimized Steffensen-Type methods with Eight-Order Convergence and High Efficiency Index, International Journal of Mathematics and Mathematical Sciences, 2012, Article ID 932420, 18 pages.

\bibitem{Soleymani5}
F. Soleymani and S.K. Khattri : Finding simple root by seventh- and eight-order derivative free methods, International Journal of Mathematical models and methods in applied sciences, Issue 1, Volume 6, 2012.
\bibitem{carlos}
C. Andreu, N. Cambil, A. Cordero and J. R. Torregrosa : A class of optimal eight-order derivative free methods for solving the Danchick-Gauss problem, Appl. Math. and Comp., 232 (2014), 237-246.
\bibitem{Gautschi}
W. Gautschi: Numerical Analysis: An Introduction Birkhauser, Barton, Mass, USA (1997).
\bibitem{Traub2}
J. F. Traub: Iterative methods for the solution of equations, Prentice Hall, New Jersey (1964).
\\






\end{thebibliography}
\end{document}